\documentclass[11pt,reqno]{amsart}
\usepackage[utf8]{inputenc}
\usepackage{amsfonts,amsthm,amsmath,amssymb,thmtools}
\usepackage{graphicx}
\usepackage{enumerate}
\usepackage{hyperref}
\usepackage{color}
\usepackage{tikz}
\usetikzlibrary{knots}
\usepackage[margin=1in]{geometry}
\usepackage[
    maxbibnames=99,
    backend=biber,
    style=alphabetic,
    giveninits=true
]{biblatex}
\DeclareFieldFormat{pages}{#1}
\renewbibmacro{in:}{%
  \ifentrytype{article}
    {}
    {\bibstring{in}%
     \printunit{\intitlepunct}}}
\DeclareFieldFormat
[article,inbook,incollection,inproceedings,patent,thesis,unpublished]
  {title}{\mkbibemph{#1}}
\DeclareFieldFormat{journaltitle}{#1\isdot}
\DeclareFieldFormat[article]{volume}{\mkbibbold{#1}}
\DeclareFieldFormat[article]{number}{\bibstring{number}\addnbspace #1}

\renewbibmacro*{journal+issuetitle}{%
  \usebibmacro{journal}%
  \setunit*{\addspace}%
  \iffieldundef{series}
    {}
    {\newunit
     \printfield{series}%
     \setunit{\addspace}}%
  \printfield{volume}%
  \setunit{\addspace}%
  \usebibmacro{issue+date}%
  \setunit{\addcomma\space}%
  \printfield{number}%
  \setunit{\addcolon\space}%
  \usebibmacro{issue}%
  \setunit{\addcomma\space}%
  \printfield{eid}
  \newunit}

\newtheoremstyle{mystyle}
  {}
  {}
  {\itshape}
  {}
  {\bfseries}
  {.}
  { }
  {\thmname{#1}\thmnumber{ #2}\thmnote{ (#3)}}

\theoremstyle{mystyle}
\newtheorem{Thm}{Theorem}
\newtheorem{Prop}[Thm]{Proposition}

\newcommand{\Z}{\mathbb{Z}}

\title{Adjunction inequality for spatially refined $s$-invariants}

\author{Qiuyu Ren}
\address{Department of Mathematics, University of California, Berkeley, Berkeley, CA 94709, USA}
\email{qiuyu\_ren@berkeley.edu}
\begin{document}

\begin{abstract}
We note an adjunction inequality in $k\overline{\mathbb{CP}^2}$ for the $s$-version of the $Sq^1$-refinement of Rasmussen's $s$-invariant. This does not hold for general spatial refinements of $s$-invariants.
\end{abstract}

\maketitle

It was proved in \cite[Corollary~1.4]{ren2024lee} that for an oriented link cobordism $\Sigma\colon L_0\to L_1$ in some $k\overline{\mathbb{CP}^2}\backslash(int(B^4)\sqcup int(B^4))$ between two oriented links $L_0,L_1\subset S^3$ with $\pi_0(L_1)\to\pi_0(\Sigma)$ surjective, we have an adjunction-type inequality for Rasmussen's $s$-invariant \cite{rasmussen2010khovanov,beliakova2008categorification} over any coefficient field $\mathbb F$ of the form
\begin{equation}\label{eq:adjunction}
	s(L_1)\le s(L_0)-\chi(\Sigma)-[\Sigma]^2-|\Sigma|.
\end{equation}

However, Ciprian Manolescu informed us that the adjunction inequality (for knots) may fail for spatially refined $s$-invariants defined by Lipshitz-Sarkar \cite{lipshitz2014refinement}, Sarkar-Scaduto-Stoffregen \cite{sarkar2020odd} and Schuetz \cite{schuetz2022two}. The counterexample he found is the knot $9_{42}$ which bounds a disk in the punctured $\overline{\mathbb{CP}^2}$ with homology class $[\overline{\mathbb{CP}^1}]$ whose geometric intersection number with $\overline{\mathbb{CP}^1}$ is $3$ (see Figure~\ref{fig:9_42}), but which has some versions of the odd $Sq^1$-, $Sq^2$- and even $Sq^2$-refined $s$-invariants being $2$, violating \eqref{eq:adjunction}.

Let us see what goes wrong with the proof in \cite{ren2024lee} if it were applied to these refined $s$-invariants. The proof relies on the following three facts about $s$-invariants.

\begin{enumerate}[(a)]
	\item $s$ is defined for all oriented links in $S^3$. If $\Sigma\colon L_0\to L_1$ is an oriented cobordism in $S^3\times I$ with $\pi_0(L_0)\to\pi_0(\Sigma)$ surjective, then $s(L_1)-s(L_0)\ge\chi(\Sigma)$.
	\item If $T(n,n)_{p,q}$ is the torus link $T(n,n)$, equipped with an orientation where $p$ strands are oriented against the other $q$ strands, $p\ge q$, $p+q=n$, then $s(T(n,n)_{p,q})=(p-q)^2-2p+1$.
	\item $s(L\sqcup T(n,n)_{p,q})=s(L)+s(T(n,n)_{p,q})-1$ for any oriented link $L$.
\end{enumerate}

Lipshitz-Sarkar's refined $s$-invariants are only defined for knots, but it is not hard to extend the definition to links and address item (a), which we will do in a while. For this natural generalization, however, one can verify using KnotJob \cite{knotjob} that $s_+^{Sq^2}(T(3,3)_{2,1})=s(T(3,3)_{2,1})+2$, violating item (b). The same violation for $T(3,3)_{2,1}$ appears for the odd $Sq^1$-refinement and the $s$-version of one of Schuetz's two odd $Sq^2$-refinements (once appropriately defined). These contribute to the failure of the corresponding refined $s$-invariants from satisfying \eqref{eq:adjunction} for the example of $9_{42}$. 

On the positive side, we show the following.
\begin{Thm}
The adjunction inequality \eqref{eq:adjunction} holds for $s_+^{Sq^1}$, the $s$-version of the even $Sq^1$-refinement of $s$ (defined for knots in \cite{lipshitz2014refinement}, and extended to all oriented links as in Section~\ref{sbsec:s_refined_links}).
\end{Thm}

For the proof of \eqref{eq:adjunction} in \cite{ren2024lee} to work, it suffices to justify (a)(b)(c). We address each of them in the following three sections, respectively.

\begin{figure}
    \centering
    \begin{tikzpicture}[thick,scale=0.8, every node/.style={scale=0.8}]
	\node (A) at (0,0) [draw,minimum width=60pt,minimum height=25pt,thick] {+1};
	\begin{knot}[
		clip width=8,
		clip radius=8pt,
		flip crossing/.list={1,2},
		]
		\strand [thick] (A.90) .. controls +(0,2.2) and +(1.1,0.1) .. (-1.8,3)
		 ..  controls +(-1.9,-0.1) and +(-2,0) .. (-2,-2.5)
		  .. controls +(1,0) and +(0,-1.4) .. (A.210);
		\strand [thick] (A.150) .. controls +(0,1) and +(0.6,0) .. (-1.6,1.8)
		to [out=left,in=left] (-1.6,-1.5)
    	to [out=right,in=left] (1.6,-1.5)
		to [out=right,in=right] (1.5,1.8)
		to [out=left,in=up] (A.30);
		\strand [thick] (A.330)
		.. controls +(0,-0.6) and  +(0.4,0) .. (0.4,-2.2)
		.. controls +(-0.4,0) and +(0,-0.6) .. (A.270);
	\end{knot}
    \end{tikzpicture}
    \caption{A diagram of the knot $9_{42}$. Untwisting the twist region results in the unknot.}
    \label{fig:9_42}
\end{figure}
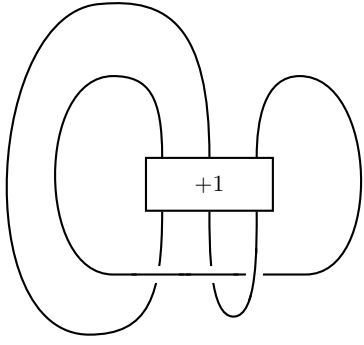

\subsection{Spatially refined \texorpdfstring{$s$}{s}-invariants for links}\label{sbsec:s_refined_links}
We will only concern ourselves with even refinements, although the odd case can be dealt with similarly. We follow \cite{lipshitz2014refinement}.

Let $\theta$ be a stable cohomology operation of degree $m$ over a field $\mathbb F$ and  $L$ be a nonempty oriented link in $S^3$ with orientation $\mathfrak o$. Let $W(L)=\mathrm{span}\{[\mathfrak s_{\mathfrak o}],[\mathfrak s_{\bar{\mathfrak o}}]\}\subset Kh_{Lee}^0(L;\mathbb F)$ (use Bar-Natan homology $Kh_{BN}$ \cite{bar2005khovanov,turner2006calculating} instead of Lee homology $Kh_{Lee}$ \cite{lee2005endomorphism} if $\mathrm{char}(\mathbb F)=2$) denote the $2$-dimensional subspace spanned by the canonical generators corresponding to the orientation of $L$ and its reverse. For an integer $q\equiv\#\pi_0(L)\pmod2$, consider the maps
\begin{equation}\label{eq:s_theta_maps}
Kh^{-m,q}(L;\mathbb F)\xrightarrow{\theta}Kh^{0,q}(L;\mathbb F)\xleftarrow{p}H^0(CKh_{Lee}^{\ge q}(L;\mathbb F))\xrightarrow{j}Kh_{Lee}^0(L;\mathbb F),
\end{equation}
where $Kh$ is the Khovanov homology \cite{khovanov2000categorification}, $CKh_{Lee}$ is the Lee chain complex, and $CKh_{Lee}^{\ge q}$ is its subcomplex consists of elements with filtered degree at least $q$.

With respect to $L$, the integer $q$ is \textit{$\theta$-half-full} (resp. \textit{$\theta$-full}) if the subspace $V^q(L):=j(p^{-1}(\mathrm{im}(\theta)))\cap W(L)$ of $W(L)$ is at least $1$-dimensional (resp. $2$-dimensional); it is \textit{half-full} (resp. \textit{full}) if the subspace $\mathrm{im}(j)\cap W(L)$ of $W(L)$ is at least $1$-dimensional (resp. $2$-dimensional). Hence
\begin{itemize}
	\item if $q$ is $\theta$-half-full (resp. $\theta$-full), then $q$ is half-full (resp. full);
	\item if $q$ is half-full (resp. full), then $q-2$ is $\theta$-half-full (resp. $\theta$-full).
\end{itemize}
Note that in this language, the classical $s$-invariant of $L$ over $\mathbb F$ is
$$s^{\mathbb F}(L)=\max\{q\colon q\text{ is half-full}\}-1=\max\{q\colon q\text{ is full}\}+1.$$

Define the $\theta$-refined $s$-invariants of $L$ by
$$r_+^\theta(L):=\max\{q\colon q\text{ is $\theta$-half-full}\}+1,\quad s_+^\theta(L):=\max\{q\colon q\text{ is $\theta$-full}\}+3.$$
Thus $r_+^\theta(L),s_+^\theta(L)\in\{s^{\mathbb F}(L),s^{\mathbb F}(L)+2\}$. If $0$ denotes the zero stable cohomology operation over $\mathbb F$, then $s^{\mathbb F}=r_+^0=s_+^0$.

For computation purpose, we note that $s_+^\theta(L)=s^{\mathbb F}(L)+2$ if and only if there exists $x\in H^0($ $CKh_{Lee}^{\ge s^{\mathbb F}(L)-1}(L;\mathbb F))$ with $j(x)=[\mathfrak s_o]$ and $p(x)\in\mathrm{im}(\theta)$. Similarly, $r_+^\theta(L)=s^{\mathbb F}(L)+2$ if and only if there exists $x\in H^0(CKh_{Lee}^{\ge s^{\mathbb F}(L)+1}(L;\mathbb F))$ with $j(x)=[\mathfrak s_o]\pm[\mathfrak s_{\bar{\mathfrak o}}]$ and $p(x)\in\mathrm{im}(\theta)$ for some sign $\pm$.

If $\Sigma\colon L_0\to L_1$ is an oriented link cobordism in $S^3\times I$ between nonempty links with $\pi_0(L_0)\to\pi_1(\Sigma)$ surjective, then $Kh_{Lee}(\Sigma)$ maps $W(L_0)$ isomorphically onto $W(L_1)$, which restricts to monomorphisms $V^q(L_0)\to V^{q+\chi(\Sigma)}(L_1)$. It follows that 
\begin{equation}\label{eq:refined_s_cobordism}
	r_+^\theta(L_1)-r_+^\theta(L_0)\ge\chi(\Sigma),\quad s_+^\theta(L_1)-s_+^\theta(L_0)\ge\chi(\Sigma).
\end{equation}
Finally, we define $r_+^\theta(\emptyset)=s_+^\theta(\emptyset)=1$. It is straightforward to check that \eqref{eq:refined_s_cobordism} still holds when $L_0,L_1$ are allowed to be the empty link.

The minus versions of the spatial refinements of $s$ are defined by $r_-^\theta(L):=-r_+^\theta(\bar L)$, $s_-^\theta(L):=-s_+^\theta(\bar L)$, where $\bar L$ denotes the mirror image of $L$. We will not be using them.

\subsection{\texorpdfstring{$Sq^1$}{Sq1}-refined \texorpdfstring{$s$}{s}-invariants for torus links}
\begin{Prop}\label{prop:s_sq1_Tnn}
	For $p\ge q$, $p+q=n$, we have $$s_+^{Sq^1}(T(n,n)_{p,q})=s^{\mathbb F_2}(T(n,n)_{p,q})=(p-q)^2-2p+1.$$ The same holds for $r_+^{Sq^1}$ when $p=q$.
\end{Prop}
\begin{proof}
	Write for short $T=T(n,n)_{p,q}$ and $s(\cdot)=s^{\mathbb F_2}(\cdot)$. By \cite[Theorem~1.1, Theorem~2.1]{ren2024lee}, we know $Kh^{0,s(T)-1}(T)=\Z$ and $Kh^{1,s(T)-1}(T)=0$. Since $Sq^1$ is the Bockstein homomorphism, it follows that $Sq^1\colon Kh^{-1,s(T)-1}(T)\to Kh^{0,s(T)-1}(T)$ is the zero map, hence $s_+^{Sq^1}(T)=s(T)$. When $p=q$, the induction scheme in \cite[Section~5]{ren2024lee} easily implies that $Kh^{-1,*}(T)=0$ for all $*$, hence $r_+^{Sq^1}(T)=s(T)$. The last equality is \cite[Theorem~1.1]{ren2024lee}.
\end{proof}

\subsection{\texorpdfstring{$s_+^{Sq^1}$}{s(Sq1)+} under disjoint union with nice links}
\begin{Prop}\label{prop:s_sq1_cntd_sum}
	Let $T$ be an oriented link in $S^3$ such that $$Sq^1\colon Kh^{i-1,s^{\mathbb F_2}(T)-1}(T;\mathbb F_2)\to Kh^{i,s^{\mathbb F_2}(T)-1}(T;\mathbb F_2)$$ is zero for $i=0,1$. Then, for any oriented link $L$ in $S^3$, we have $$s_+^{Sq^1}(L\sqcup T)=s_+^{Sq^1}(L)+s_+^{Sq^1}(T)-1.$$
\end{Prop}
By arguments in the proof of Proposition~\ref{prop:s_sq1_Tnn} above, we know all torus links $T=T(n,n)_{p,q}$ satisfy the conditions in Proposition~\ref{prop:s_sq1_cntd_sum}.
\begin{proof}
Write for short $s(\cdot)=s^{\mathbb F_2}(\cdot)$. When $L=\emptyset$ or $T=\emptyset$ the statement is trivial. Below we assume that both $T,L$ are nonempty. Since $Sq^1\colon Kh^{-1,s(T)-1}(T;\mathbb F_2)\to Kh^{0,s(T)-1}(T;\mathbb F_2)$ is zero, we know $s_+^{Sq^1}(T)=s(T)$.
	
Since $s(L\sqcup T)=s(L)+s(T)-1$, it suffices to show that $s_+^{Sq^1}(L\sqcup T)=s(L\sqcup T)+2$ if and only if $s_+^{Sq^1}(L)=s(L)+2$. Although the ``if'' direction is simpler, we prove the two directions together.
	
We first observe that any filtered cochain complex $C^*$ over a field $\mathbb F$ is filtered isomorphic to a filtered complex of the form $$C^*\cong H^*(C)\oplus(\oplus_\alpha E_\alpha^*)\oplus(\oplus_\beta A_\beta^*)$$ where each $E_\alpha$ (resp. $A_\beta$) is a $2$-term filtered cochain complex $\mathbb F\xrightarrow{=}\mathbb F$ in some homological degree, which is not (resp. is) filtration-preserving. Thus, $C^*$ is filtered chain homotopy equivalent to $H^*(C)\oplus(\oplus_\alpha E_\alpha^*)$.
	
In particular, for each $X\in\{L,T\}$, we can choose such a filtered chain homotopy equivalence (which we simply denote as $=$) $CKh_{BN}^*(X)=Kh_{BN}^*(X)\oplus(\oplus_\alpha E_{X,\alpha}^*)$. Then $Kh^{*,*}(X;\mathbb F_2)=gr^*(Kh_{BN}^*(X))\oplus(\oplus_\alpha gr^*(E_{X,\alpha}^*))$, and $p_X(j_X^{-1}([\mathfrak s_{\mathfrak o_X}])\subset Kh^{0,s(X)-1}(X;\mathbb F_2)$ (see \eqref{eq:s_theta_maps} for notations; subscripts indicate dependence on the link) is the affine space over $$V_X:=\bigoplus_\alpha\mathrm{im}(H^0(E_{X,\alpha}^{*,\ge s(X)-1})\to gr^{s(X)-1}(E_{X,\alpha}^0))$$ containing $x_X:=gr^{s(X)-1}([\mathfrak s_{\mathfrak o_X}])$.
	
Similarly, for the disjoint union, from
\begin{align*}
CKh_{BN}^*(L\sqcup T)=Kh_{BN}^*(L)\otimes Kh_{BN}^*(T)&\,\oplus(\oplus_\alpha E_{L,\alpha}^*\otimes Kh_{BN}^*(T))\\\oplus&(\oplus_\beta Kh_{BN}^*(L)\otimes E_{T,\beta}^*)\oplus(\oplus_{\alpha,\beta}E_{L,\alpha}^*\otimes E_{T,\beta}^*)
\end{align*}
we know that $p_{L\sqcup T}(j_{L\sqcup T}^{-1}([\mathfrak s_{\mathfrak o_L}]\otimes[\mathfrak s_{\mathfrak o_T}]))\subset Kh^{0,s(L\sqcup T)-1}(L\sqcup T;\mathbb F_2)$ is the space containing $x_L\otimes x_T$ affine over $$V_{L\sqcup T}:=V_L\otimes gr^{s(T)-1}(Kh_{BN}^0(T))\oplus W$$ for some subspace $W\subset H_{skew}$, where 
\begin{align*}
H_{skew}:=&\,(\oplus_{(h,q)\ne(0,s(T)-1)}Kh^{-h,s(L\sqcup T)-1-q}(L;\mathbb F_2)\otimes Kh^{h,q}(T;\mathbb F_2))\\&\oplus Kh^{0,s(L)-1}(L;\mathbb F_2)\otimes(\oplus_\beta gr^{s(T)-1}(E_{T,\beta}^0)))
\end{align*}
is a complement of $Kh^{0,s(L)-1}(L;\mathbb F_2)\otimes gr^{s(T)-1}(Kh_{BN}^0(T))$ in $Kh^{0,s(L\sqcup T)-1}(L\sqcup T;\mathbb F_2)$.
	
On the other hand, since $Sq_{L\sqcup T}^1=Sq_L^1\otimes1+1\otimes Sq_T^1$, by our assumption on $T$, the image $\mathrm{im}(Sq_{L\sqcup T}^1)\subset Kh^{0,s(L\sqcup T)-1}(L\sqcup T;\mathbb F_2)$ is equal to $$\mathrm{im}(Sq_L^1\colon Kh^{-1,s(L)-1}(L;\mathbb F_2)\to Kh^{0,s(L)-1}(L;\mathbb F_2))\otimes gr^{s(T)-1}(Kh_{BN}^0(T))\oplus W'$$ for some subspace $W'\subset H_{skew}$. It follows that 
\begin{align*}
&\,s_+^{Sq^1}(L\sqcup T)=s(L\sqcup T)+2\\
\iff&\,x_L\otimes x_T\in V_{L\sqcup T}+\mathrm{im}(Sq_{L\sqcup T}^1)\\
\iff&\,x_L\in V_L+\mathrm{im}(Sq_L^1)\\
\iff&\,s_+^{Sq^1}(L)=s(L)+2.\hfill\qedhere
\end{align*}
\end{proof}

\subsection*{Acknowledgment}
We thank Ciprian Manolescu and Robert Lipshitz for helpful discussions and correspondence.

\printbibliography

\end{document}